\documentclass[12pt]{amsart}
\usepackage{amsfonts,amsthm,amsmath,amssymb,verbatim}
\usepackage{graphicx}

\newtheorem{thm}{Theorem}[section]

\newtheorem{lemma}[thm]{Lemma}

\theoremstyle{remark}

\newtheorem{example}[thm]{Example}
\newtheorem{remark}[thm]{Remark}
\newtheorem{defin}{Definition}



\def\C{\mathbb{C}}

\def\Z{\mathbb{Z}}
\def\P{\mathbb{P}}

\def\R{\mathbb{R}}

\def\PT{\Pi}

\def\a{\alpha}

\def\l{\lambda}

\def\L{\Lambda}

\def\w0{\overline{w_0}}
\def\id{\rm id}

\def\emptyset{\varnothing}

\title[N.--O. polytopes of B.--S. varieties as Minkowski sums]
{Newton--Okounkov polytopes of Bott--Samelson varieties as Minkowski sums}
\author{Valentina Kiritchenko}
\email{vkiritch@hse.ru}
\thanks{The study has been funded by the Russian Academic Excellence Project '5-100'.}

\address{Laboratory of Algebraic Geometry and Faculty of Mathematics\\
National Research University Higher School of Economics, Russian Federation\\
Usacheva str. 6, 119048 Moscow, Russia}
\address{Institute for Information Transmission Problems, Moscow, Russia}

\date{}

\begin{document}
\begin{abstract}
We compute the Newton--Okounkov bodies of line bundles on a Bott--Samelson resolution
of the complete flag variety of $GL_n$ for a geometric valuation coming from a flag of
translated Schubert subvarieties.
The Bott--Samelson resolution corresponds to the decomposition
$(s_1)(s_2s_1)(s_3s_2s_1)(\ldots)(s_{n-1}\ldots s_1)$
of the longest element in the Weyl group, and the Schubert subvarieties correspond to the
terminal subwords in this decomposition.
We prove that the resulting Newton--Okounkov polytopes for semiample line bundles satisfy
the additivity property with respect to the Minkowski sum.
In particular, they are Minkowski sums of Newton--Okounkov
polytopes of line bundles on the complete flag varieties for $GL_2$,\ldots, $GL_{n}$.
\end{abstract}

\maketitle

\section{Introduction}
Newton--Okounkov convex bodies provide a tool for extending toric geometry to non-toric varieties.
For instance, Newton--Okounkov polytopes of flag varieties were used to build a positive convex
geometric model for Schubert calculus where intersection of faces corresponds
to the intersection product  of cycles \cite{K10,KST,Ki16I}.
Another motivation to study such polytopes comes from representation theory.
String polytopes of Berenstein--Zelevinsky and Littelmann (in particular,
Gelfand--Zetlin polytopes), Feigin--Fourier--Littelmann--Vinberg polytopes and Nakashima--Zelevinsky
polyhedral realizations were exhibited as Newton--Okounkov polytopes of flag varieties for certain
geometric valuations \cite{FFL14,FFL15,Ka,K15,FO}.

An essential feature of Newton polytopes of toric varieties is the additivity property with respect
to the Minkowski sum.
Let $\Delta_v(X,L)$ denote the Newton--Okounkov convex body of a line bundle $L$ on a variety $X$
with respect to a valuation $v$ on the field $\C(X)$ of rational functions.
By the {\em additivity property} we mean that
$$\Delta_v(X,L_1\otimes L_2)=\Delta_v(X,L_1)+\Delta_v(X,L_2)$$
for any two semiample line bundles $L_1$ and $L_2$ on $X$ (cf. \cite{KaKh}).
If $X$ is toric, then there is a natural valuation $v_0$ on $\C(X)$ that assigns to every Laurent
polynomial its lowest degree term.
The resulting polytopes $\Delta_{v_0}(X,L)$ are called {\em Newton polytopes} and satisfy the
additivity property.
In particular, this property is used in the famous Bernstein--Koushnirenko theorem
to identify intersection indices of divisors in toric varieties with mixed volumes of Newton
polytopes.
In general, the additivity property does not necessarily hold for Newton--Okounkov convex bodies
(see Section \ref{ss.NO} for more details).
A natural problem for non-toric varieties $X$ is to find a valuation $v_0$ on $\C(X)$ such that
the resulting Newton--Okounkov bodies $\Delta_{v_0}(X,L)$ still satisfy the additivity property.

In this preprint, we establish additivity property for a geometric valuation on Bott--Samelson
varieties considered in \cite{An,K15}.
Recall that Bott--Samelson resolutions $R_I$ of Schubert varieties depend on the choice of a
sequence $I$ of simple roots
and are constructed as towers of successive $\P^1$-fibrations
(see Section \ref{ss.BS} for a reminder).
If $I$ encodes a reduced decomposition $w_I$ of the longest word in the Weyl group
then $R_I$ is birationally isomorphic to the complete flag variety $G/B$, that is,
$\C(R_I)\simeq\C(G/B)$.
Let $d$ denote the dimension of $G/B$.
In \cite{K15}, we dealt with a geometric valuation $v_0$ on $\C(G/B)$ coming from the flag of
translated Schubert subvarieties
$$w_0X_{\id}\subset w_0w_{d-1}^{-1}X_{w_{d-1}}\subset w_0w_{d-2}^{-1}X_{w_{d-2}}
\subset\ldots\subset w_0w_{1}^{-1}X_{w_{1}}\subset X,$$
where $w_1$, $w_2$,\ldots, $w_{d-1}$ are terminal subwords of the decomposition $w_I$ (see Section
\ref{ss.NO} for a precise definition).
We now consider the same valuation on the Bott--Samelson variety $R_I$.

Let $G=GL_n$, and $I=(\a_1;\a_2,\a_1;\a_3,\a_2,\a_1;\ldots;\a_{n-1},\ldots,\a_{1})$ where $\a_1$, \ldots, $\a_{n-1}$ are simple roots.
Then
$$w_I=(s_1)(s_2s_1)(s_3s_2s_1)(\ldots)(s_{n-1}\ldots s_1).$$
In \cite{K15}, valuation $v_0$ was used to compute Newton--Okounkov polytopes
$\Delta_{v_0}(R_I,L)$ for all semiample line bundles $L$ coming from the complete flag variety
$GL_n/B_n$.
In the present preprint, we extend this computation to all semiample line bundles on $R_I$.
By construction, $R_I$ admits projections to $n-1$ flag varieties $GL_2/B_2$,\ldots,
$GL_{n}/B_{n}$, and pullbacks of semiample line bundles from these flag varieties span the
semigroup of semiample divisors on $R_I$ \cite[Corollary 3.3]{LT}.
\begin{thm} \label{t.main}
Let $L$ be a semiample line bundle on $R_I$, and $L=L_1\otimes\cdots\otimes L_{n-1}$
its decomposition into line bundles coming from the flag varieties $GL_2/B_2$,\ldots, $GL_{n}/B_n$.
Then
$$\Delta_{v_0}(R_I,L)=\Delta_{v_0}(R_I,L_1)+\ldots+\Delta_{v_0}(R_I,L_{n-1}),$$
that is,  Newton--Okounkov polytopes satisfy additivity property.
\end{thm}
Together with \cite[Theorem 2.1]{K15} this theorem yields a description of
Newton--Okounkov polytopes for all semiample line bundles on $R_I$ as Minkowski sums of
Feigin--Fourier--Littelmann--Vinberg polytopes.
For $n=3$, this agrees with computation in \cite[Section 6.4]{An} (see Example \ref{e.An}).
The proof of Theorem \ref{t.main} uses divided difference operators on polytopes defined in
\cite{Ki16II}, in particular, their additivity with respect to the Minkowski sum.

Note that Bott--Samelson varieties are topologically the same as smooth toric varieties
whose Newton polytopes are multidimensional trapezoids called {\em Grossberg--Karshon cubes}
\cite{GK}.
There are valuations on $R_I$ that produce Grossberg--Karshon cubes as Newton--Okounkov polytopes
of $R_I$ for certain class of very ample line bundles on $R_I$ \cite{HY,Fu}.
However, these polytopes do not satisfy additivity property already for
$R_{121}$ in the $GL_3$ case (cf. \cite[Examples 4.1-4.3]{HY} and \cite[Example B1]{Fu}).


\section{Reminder on Bott--Samelson varieties, Newton--Okounkov convex bodies and
Feigin--Fourier--Littelmann--Vinberg polytopes}

In this section, we recall the definition of the Bott--Samelson variety $R_I$, and describe its
Picard group using \cite[Theorem 3.1]{LT}.
We also recall the definition of Newton--Okounkov convex bodies and their
superadditivity property \cite[Proposition 2.32]{KaKh}.
Finally, we recall an elementary definition of Feigin--Fourier--Littelmann--Vinberg (FFLV)
polytopes in type $A$ following \cite{FFL}.

\subsection{Bott--Samelson varieties}
\label{ss.BS} Let $G$ be a connected complex reductive group with simple roots $\a_1$, \ldots, $\a_r$,
and $B\subset G$ a Borel subgroup.
Denote by $P_i$ the minimal parabolic subgroup corresponding to the root $\a_i$.
Let $\pi_i:G/B\to G/P_i$ be the projection of
the complete flag variety to the partial flag variety.
In what follows, we mostly deal with the case $G=GL_n(\C)$.
In this case, $r=n-1$ and
$G/B=\{(V^1\subset V^2\subset\ldots\subset V^{n-1}\subset \C^n)\}$ is the variety of complete
flags in $\C^n$.
The projection $\pi_i$ to the variety of partial flags
$G/P_i=\{(V^1\subset\ldots \subset V^{i-1}\subset V^{i+1}\subset\ldots\subset V^{n-1}\subset \C^n)\}$
forgets the $i$-th subspace $V^i$.

Let $I=(i_1,\ldots,i_\ell)$ denote a sequence of numbers such that $i_j\in\{1,\ldots,r\}$, that is,
$I$ defines a sequence of simple roots $(\a_{i_1},\ldots, \a_{i_\ell})$.
A {\em Bott--Samelson variety} $R_I$ together with a map
$r_I:R_I\to G/B$ is defined inductively as a tower of
successive $\P^1$-fibrations:
$$R_{\emptyset}\leftarrow R_{(i_1)}\leftarrow R_{(i_1,i_2)}\leftarrow\ldots\leftarrow R_{I^\ell}\leftarrow R_I,$$
where $I^{\ell}$ denotes $(i_1,\ldots,i_{\ell-1})$.
Put $R_{\emptyset}=\{pt\}$.
Assume that $r_{I^{\ell}}:R_{I^{\ell}}\to X$ is already defined and define
$R_I$ as the fiber product $R_{I^\ell}\times_{G/P_{i_\ell}} G/B$, that is,
$$R_I=\{(x,y)\in R_{I^\ell}\times G/B\ |\ r_{I^{\ell}}(x)=\pi_{i_\ell}(y)\}.$$
The map $r_I:R_I\to X$ is defined as the projection to the second factor.
By construction, $\dim X_I=\ell$.

Alternatively, $R_I=\P(r_{I^\ell}^*\pi_{i_\ell}^*E_{i_\ell})$ where $E_i$ is a rank two
vector bundle on $G/P_i$ such that $\P(E_i)=G/B$.
For instance, one can take $E_i={\pi_i}_*L_{\rho}$ as a uniform choice for all $i$.
Here $\rho$ denotes the sum of all dominant weights of $G$, and $L_\rho$ the line bundle on $G/B$
corresponding to $\rho$.

It is easy to check that $r_I:R_I\to G/B$ is a resolution of singularities for the Schubert variety
$X_{w_I}$ whenever $w_I:=s_{i_1}\ldots s_{i_\ell}$ is reduced.
Here $s_i$ denotes the simple reflection corresponding to the root $\a_i$.
For $GL_n$, we identify simple reflections with elementary transpositions, that is, $s_i:=(i~~i+1)$.
By construction, the projection $r_I:R_I\to X_{w_I}$ gives an isomorphism between the open Schubert cell
$U_{w_I}$ in $X_{w_I}$ and its preimage $r_I^{-1}(U_{w_I})$ in $R_I$.
In general, $\pi(R_I)$ is the Schubert variety corresponding to the
{\em Demazure product} of $s_{i_1}$,\ldots, $s_{i_\ell}$ (see \cite{An15}
for more details).
In particular, if $w_I$ is not reduced then $\dim r_I(R_I)<\ell$.

\begin{example} \label{e.R_121} Let $n=3$ and $I=(1,2,1)$.
Identify subspaces in $\C^3$ with their projectivizations in $\C\P^2$.
Then $R_\emptyset\hookrightarrow G/B$ can be thought of as a fixed flag $(a_0\in\l_0\subset\P^2)$,
where $a_0$ is a point and $l_0$ is a line.
Hence, $R_1=\{a_1\in\P^2|a_1\in l_0\}$, and $R_{12}=\{(a_1\in l)\subset\P^2|a_1\in l_0, a_1\in l\}$.
Finally, the Bott--Samelson variety $R_I$ consists of all triples $(a_1,l,a_2)$, where $a_1$, $a_2$
are points and $l$ is a line such that $a_1,a_2\in l$ and $a_1\in l_0$.
The projection $R_I\to G/B$ forgets $a_1$.
\end{example}
In the case $G=GL_n$, fix the decomposition
$$\w0=(s_1)(s_2s_1)(s_3s_2s_1)\ldots(s_{n-1}\ldots s_1)$$
of the longest
element $w_0\in S_n$.
In what follows, $X_{\w0}$ denotes the Bott--Samelson variety $X_{I(n)}$ for
$I(n):=(1;2,1;3,2,1;\ldots;n-1,\ldots,1)$.
Since $\w0$ is reduced, the map $r_{I(n)}:X_{\w0}\to G/B$ is a birational isomorphism.
Similarly to Example \ref{e.R_121}, points of $X_{\w0}$ can be identified with
collections of subspaces in $\C^n$ that satisfy certain incidence relations
(see \cite[Section 2.2]{K15}).

By construction, the Bott--Samelson variety $X_{\w0}$ admits $(n-1)$ projections
$p_1$,\ldots, $p_{n-1}$ to flag varieties.
Indeed, by definition $I(n)$ contains $I(n-i+1)$ as an initial subword, hence, there is a projection
$R_{I(n)}\to R_{I(n-i+1)}$.
Define $p_i:X_{\w0}\to GL_{n-i+1}/B_{n-i+1}$ as the composition of this projection with
$r_{I(n-i+1)}$.
Recall that the Picard group of the flag variety $GL_{n}/B_{n}$  is spanned by the line bundles
$L_{\L}$ corresponding to the dominant weights $\L$ of $GL_n$ (see \cite[Remark 1.4.2]{B}
for more details).
Using \cite[Corollary 3.3]{LT} and induction it is easy to prove the following lemma.

\begin{lemma}\label{l.Picard}
The semigroup of semiample divisors on $X_{\w0}$ is spanned by $p_1^*L_{\L_1}$,\ldots,
$p^*_{n-1}L_{\L_{n-1}}$ where $\L_i$ runs through dominant weights of $GL_{n-i+1}$.
\end{lemma}

\subsection{Newton--Okounkov convex bodies}\label{ss.NO}
Let $X$ be a projective variety of dimension $d$, and $L$ a very ample line bundle on $X$.
By fixing a global section $s_0$ of $L$ we can identify the space of global sections $H^0(X,L)$
with a subspace of the field of rational functions $\C(X)$.
Let $v:\C(X)\setminus\{0\}\to\Z^d$ be a surjective valuation.
For instance, one can choose local coordinates $x_1$,\ldots, $x_d$ on $X$ and assign to every
polynomial its lowest order term with respect to some ordering on $\Z^d$.
More geometrically, take a full flag of subvarieties
$\{x_0\}=X_0\subset X_1\subset\ldots\subset X_d=X$ at a smooth point $x_0\in X$ and assign to every
rational function its (properly defined) orders of vanishing along $X_i$ considered as
a hypersurface in $X_{i+1}$ (see \cite[Examples 2.12, 2.13]{KaKh} for more details).

\begin{example}\label{e.BS} Our main example will be the Bott--Samelson variety $X=R_I$ under
the assumption that $w_I:=s_{i_1}\ldots s_{i_\ell}$ is reduced.
By construction $R_I$ contains an affine space $U\simeq \C^\ell$, namely, the preimage of
the open Schubert cell $U_{w_I}\subset G/B$ under the projection $R_I\to X_{w_I}$.
Denote by $w_k:=s_{i_k}\ldots s_{i_\ell}$ the $k$-th terminal subword of $w_I$.
Consider the flag
$$w_I U_{\id}\subset w_Iw_{\ell-1}^{-1}U_{w_{\ell-1}}\subset w_Iw_{\ell-2}^{-1}U_{w_{\ell-2}}
\subset\ldots\subset w_Iw_{1}^{-1}U_{w_{1}}\subset U_{w_I},$$
of translated Schubert cells.
It defines the lowest term valuation on the open Schubert cell $U_{w_I}$, and hence on its
preimage $U$ in $R_I$.
Denote this valuation by $v_I$.
\end{example}

In what follows, we focus on the case $G=GL_n$ and $I=I(n)$
(that is, $R_I=X_{\w0}$).
\begin{defin} \label{d.val} For $G=GL_n$ and $\w0=(s_1)(s_2s_1)(s_3s_2s_1)\ldots(s_{n-1}\ldots s_1)$,
define $v_0$ as the valuation $v_{I(n)}$ on the Bott--Samelson variety $X_{\w0}=R_{I(n)}$.
That is, $v_0$ is the lowest term valuation corresponding to the
flag of translated Schubert cells (defined in Example
\ref{e.BS}) for terminal subwords of $\w0$.
\end{defin}
A more explicit definition of $v_0$
using natural geometric coordinates on $X_{\w0}$ can be found in \cite[Section 2.2]{K15}.

The {\em Newton--Okounkov convex body} $\Delta_v(X,L)$ is defined as the closure of the
convex hull of the set
$\bigcup_{k=1}^\infty\{\frac{v(s/s_0^k)}{k} \ |\ s\in H^0(X,L^{\otimes k}) \}\subset\Z^d\subset\R^d$.
Explicit description of $\Delta_v(X,L)$ (e.g., by inequalities) is usually a challenging task.
Sometimes, it is enough to compute
$$\Delta^1_v(X,L)={\rm conv}\{v(s/s_0)\ | \ s\in H^0(X,L) \},$$
that is, the first polytope approximation of $\Delta_v(X,L)$.
By \cite[Corollary 3.2]{KaKh}, we have that $\Delta^1_v(X,L)=\Delta_v(X,L)$ whenever the volume of
$\Delta^1_v(X,L)$ times $d!$ coincides with the degree of $X$ embedded into $\P(H^0(X,L)^*)$
(this argument will be used in the proof of Theorem \ref{t.main}).
Note that if $X$ is a toric variety and $v$ is any valuation defined using standard
coordinates on the open torus orbit $(\C^*)^d\subset X$, then $\Delta_v(X,L)$ coincides
with the classical Newton (or moment) polytope of $X$.

Newton-Okounkov polytopes satisfy the following {\em superadditivity property}:
$$\Delta_v(X,L_1)+\Delta_v(X,L_2)\subset \Delta_v(X,L_1\otimes L_2) $$
(see \cite[Proposition 2.32]{KaKh} for more details).
In general, it is not true that
$$\Delta_v(X,L_1)+\Delta_v(X,L_2)=\Delta_v(X,L_1\otimes L_2).$$
For instance, if $X=R_{121}$ is the Bott-Samelson variety from Example \ref{e.BS}, and $v$ is either
a valuation considered in \cite{HY} or the one in \cite{Fu}, then
$\Delta_v(R_{121},L_1)+\Delta_v(R_{121},L_2)\ne\Delta_v(X,L_1\otimes L_2)$
for $L_1$ and $L_2$ that come from $GL_2/B_2$ and $GL_3/B_3$
(see \cite[Examples 4.1-4.3]{HY} and \cite[Example B1]{Fu} for more details).
Some general results on representation of Newton--Okounkov poytopes as Minkowski sums can be found in
\cite{SchS}.
\subsection{FFLV polytopes}
Recall that a dominant weight of $GL_n$ is  a non-decreasing collection of integers
$\L:=(\l_1,\ldots,\l_n)\in\Z^n$, that is, $\l_1\ge\l_2\ge\ldots\ge \l_n$.
For every dominant $\L$, we now define the FFLV polytope $FFLV(\L)$.
Put $d:=\frac{n(n-1)}2$
Label coordinates in $\R^d$ by
$(u^1_{n-1};u^2_{n-2},u^1_{n-2};\ldots;u^{n-1}_1,u^{n-2}_{1},\ldots,u^{1}_1)$.
Arrange the coordinates into the table
$$
\begin{array}{cccccccccc}
\l_1&       & \l_2    &         &\l_3          &    &\ldots    & &       &\l_n   \\
    &u^1_1&         &u^1_2  &         & \ldots   &       &  &u^1_{n-1}&       \\
    &       &u^2_1 &       &  \ldots &   &        &u^2_{n-2}&         &       \\
    &       &       &  \ddots   & &  \ddots   &      &         &         &       \\
    &       &       &  &u^{n-2}_1&     &  u^{n-2}_2 &        &         &       \\
    &       &         &    &     &u^{n-1}_1&   &              &         &       \\
\end{array}
\eqno{(FFLV)}$$
The polytope $FFLV(\L)$ is defined by inequalities
$u^l_m\ge 0$ and
$$\sum_{(l,m)\in D}u^l_m\le \l_i-\l_j$$
for all Dyck paths going from $\l_i$ to $\l_j$ in table $(FFLV)$  where $1\le i<j\le n$
(see \cite{FFL} for more details).

In what follows, we consider simultaneously several FFLV polytopes for different $n$.
To do this we use a flag of subspaces
$$\R^1\subset\R^3\subset\ldots\subset\R^{d_i}\subset\ldots\subset\R^d,$$
where $d_i:=\frac{(n-i+1)(n-i)}{2}$, and $\R^{d_i}$ is the coordinate subspace given by
vanishing of the first $(n-1)+\ldots+(n-i+1)$ coordinates, that is, by equations
$u^1_1=\ldots=u^{n-i+1}_{i-1}=0$ if $i>1$.
If $i=1$, then $d_1=d$ and $\R^{d_1}$ is the whole space $R^d$.
We assume that $FFLV(\L_i)$ for a dominant weight $\L_i$ of $GL_{n-i+1}$ lives in the subspace
$\R^{d_i}$.

\section{Main result}\label{s.main}
We now describe Newton--Okounkov polytopes of semiample line bundles on the Bott--Samelson variety
$X_{\w0}$ for the lowest term valuation $v_0$ (see Definition \ref{d.val}).
By Lemma \ref{l.Picard}, any semiample line bundle $L$ on $X_{\w0}$ is the tensor product
$L(\L_1,\ldots,\L_{n-1}):=p_1^*L_{\L_1}\otimes\ldots\otimes p^*_{n-1}L_{\L_{n-1}}$ where
$L_{\L_i}$ is the semiample line bundle on the flag variety $GL_{n-i+1}/B_i$ corresponding to a
dominant weight $\L_i$.

\begin{thm}
The Newton--Okounkov polytope $\Delta_{v_0}(X_{\w0},L(\L_1,\ldots,\L_{n-1}))$ is equal to the
Minkowski sum of FFLV polytopes
$$FFLV(\L_1)+FFLV(\L_2)+\ldots+FFLV(\L_{n-1}).$$
\end{thm}
It would be interesting to describe explicitly the inequalities that define the Minkowski sum
of FFLV polytopes in the theorem, e.g., by using Dyck paths that do not necessarily
start at the first row.

\begin{example}{cf. \cite[Section 6.4]{An}}\label{e.An}
Let $n=3$.
If $\L_1=(\l^1_1,\l^1_2,\l^1_3)$ and $\L_2=0$, then the 3-dimensional polytope
$\Delta_v(X_{\w0},L(\L_1,\L_2))$
is given by 6 inequalities
$$0\le u^1_1\le\l_1^1-\l_2^1; \quad 0\le u^1_2\le\l_2^1-\l_3^1; \quad 0\le u^2_1; \quad
u^1_1+u^2_1+u^1_2\le\l_1^1-\l_3^1.$$
If $\L_1=0$ and $\L_2=(\l^2_1,\l^2_2)$, then $\Delta_v(X_{\w0},D(\L_1,\L_2))$ is the
segment
$$0\le u^1_2\le \l^2_2-\l^2_1; \quad u^1_1=u^2_1=0.$$
In general, the polytope $\Delta_v(X_{\w0},L(\L_1,\L_2))$
is the Minkowski sum of the polytope $FFLV(\L_1)$ and the segment $FFLV(\L_2)$.
It is given by 7 non-redundant inequalities if $\L_1$ and $\L_2$ are strictly dominant.

For $\L_1=(1,0,-1)$, $\L_2=(0,0)$ and for $\L_1=(1,0,-1)$, $\L_2=(1,0)$
we get the polytopes depicted on \cite[Figure 3(a)]{An} and \cite[Figure 3(b)]{An}, respectively.
\end{example}

\subsection{Proof of Theorem \ref{t.main}}
\label{s.proof}
The main ingredient of the proof is the construction of Gelfand--Zetlin (GZ) polytopes using
convex geometric divided difference operators $D_1$,\ldots, $D_{n-1}$ (see \cite[Theorem 3.4]{Ki16II}
for the details).
These operators act on convex polytopes in $\R^d$ by mimicking Demazure operators.
By definition they satisfy the
following additivity property with respect to Minkowski sum of polytopes:
$$D_i(P+Q)=D_i(P)+D_i(Q)$$
for any pair of polytopes $P$ and $Q$ for which $D_i(P)$ and $D_i(Q)$ are well-defined
(results of \cite{Ki16II} are formulated in the more general context of convex chains but here
we will restrict ourselves to the case of polytopes).
In what follows, we use notation and definitions of \cite{Ki16II}.

By the superadditivity property of Newton--Okounkov convex bodies we have the inclusion
$$\Delta_{v_0}(GL_n/B_n,L_{\L_{1}})+\ldots+\Delta_{v_0}(GL_2/B_2,L_{\L_{n-1}})
\subset\Delta_{v_0}(X_{\w0},L(\L_1,\ldots,\L_{n-1})).$$
Since $\Delta_{v_0}(GL_{n-i+1}, L_{\L_i})=FFLV(\L_i)$ by \cite[Theorem 2.1]{K15} this
implies the inclusion
$$FFLV(\L_1)+FFLV(\L_2)+\ldots+FFLV(\L_{n-1})\subset\Delta_{v_0}(X_{\w0},L(\L_1,\ldots,\L_{n-1})).$$
By \cite[Corollary 3.2]{KaKh}, 
the above inclusion of polytopes is equality if and only if the
volume of the left hand side times $d!$ is equal to the degree of $X_{\w0}$ in the embedding given
by the line bundle $L(\L_1,\ldots,\L_{n-1})$.

To complete the proof we compare the degree and the volume in two steps.
First, we identify the degree of $X_{\w0}$ with the normalized volume of
the Minkowski sum of GZ polytopes using generalized Demazure theorem.
Second, we compare volumes of Minkowski sums of FFLV and GZ polytopes.
These steps are conducted in Lemmas \ref{l.degree} and \ref{l.comparison}, respectively.
\begin{lemma}\label{l.degree}
The volume of the Minkowski sum $GZ(\L_1)+GZ(\L_2)+\ldots+GZ(\L_{n-1})$
of GZ polytopes is equal to $d!$ times the degree of the Bott--Samelson variety $X$
in the embedding given by the line bundle $L(\L_1,\ldots,\L_{n-1})$ (assuming that the latter
is very ample).
\end{lemma}
\begin{proof} For every dominant weight $\L_i=(\l_1^i,\ldots,\l_{n-i+1}^i)$ of $GL_{n-i+1}$
define the point $a_{\L_i}=(0,\ldots,0, \l_2^i,\ldots,\l_{n-i+1}^i; \l_3^i,\ldots,\l_{n-i+1}^i;
\ldots; \l_{n-i+1}^i)\in\R^{d_i}\subset\R^d$.
The additivity property of $D_i$ implies immediately the following
identity:
$$D_1(D_2D_1)(\ldots)(D_{n-1}\ldots D_1)(a_{\L_1})+D_1(D_2D_1)(\ldots)(D_{n-2}\ldots D_1)(a_{\L_2})+$$
$$\ldots+D_1(D_2D_1)(a_{\L_{n-2}})+D_1(a_{\L_{n-1}})=$$
$$=D_1\PT_{\L_{n-1}}(D_2D_1)\PT_{\L_{n-2}}(\ldots)\PT_{\L_2}(D_{n-1}\ldots D_1)(a_{\L_1}),$$
where $\PT_v$ is the operator of the parallel transport by a vector $v$, in particular, $\PT_v(P)=v+P$
for any polytope $P\subset\R^d$.

Combining the above identity with \cite[Theorem 3.4]{Ki16I} we can interpret the Minkowski sum of
GZ polytopes as follows:
$$GZ(\L_1)+GZ(\L_2)+\ldots+GZ(\L_{n-1})=
D_1\PT_{\L_{n-1}}(D_2D_1)\PT_{\L_{n-2}}(\ldots)\PT_{\L_2}(D_{n-1}\ldots D_1)(a_{\L_1}).$$

Hence, by \cite[Theorem 3.2]{Ki16I} the character of $GZ(\L_1)+GZ(\L_2)+\ldots+GZ(\L_{n-1})$
coincides with
$$T_1e^{\L_{n-1}}(T_2T_1)e^{\L_{n-2}}(\ldots)e^{\L_2}(T_{n-1}\ldots T_1)e^{\L_1},$$
where $T_i$ denotes the Demazure operator corresponding to the $i$-th simple root of $GL_n$.
By generalized Demazure theorem \cite[Theorem 5]{LLM} this is exactly the Demazure character of the
$B$-module $H^0(X_{\w0}, L(\L_1,\ldots,\L_{n-1}))^*$.
In particular, the Ehrhart polynomial of $GZ(\L_1)+GZ(\L_2)+\ldots+GZ(\L_{n-1})$
coincides with the Hilbert polynomial of $X_{\w0}$ in the embedding
given by the line bundle $L(\L_1,\ldots,\L_{n-1})$.
This implies the statement of the lemma.
\end{proof}

\begin{lemma}\label{l.comparison}
The Minkowski sums $GZ(\L_1)+GZ(\L_2)+\ldots+GZ(\L_{n-1})$ and
$FFLV(\L_1)+FFLV(\L_2)+\ldots+FFLV(\L_{n-1})$ have the same volume.
\end{lemma}
\begin{proof}
It is not hard to extend arguments of \cite[Section 4]{Ki16I} on comparison of FFLV and
GZ polytopes so
that they also work for Minkowski sums of these polytopes.
Namely, one can show that the polytopes 
$GZ(\L_1,\ldots,\L_{n-1}):=GZ(\L_1)+\ldots+GZ(\L_{n-1})$ and
$FFLV(\L_1,\ldots,\L_{n-1}):=FFLV(\L_1)+\ldots+FFLV(\L_{n-1})$ have the same Ehrhart 
polynomial by constructing both of them recursively as in \cite[Section 4]{K15}. 
\end{proof}

\end{document}